\documentclass{amsart}
\usepackage{amssymb}
\usepackage{amsfonts}

\theoremstyle{plain}
\numberwithin{equation}{section}

\newtheorem {YouHesa}{Theorem}[section]
\newtheorem {Sahni}[YouHesa]{Theorem}
\newtheorem {singhsingh}{Theorem}[section]

\newtheorem {singthuk}[singhsingh]{Lemma}

\newtheorem{hpdecomp}[singhsingh]{Lemma}

\newtheorem{sinagra2}[singhsingh]{Lemma}
\newtheorem {blaschke}{Lemma}[section]
\newtheorem {firstmain}[blaschke]{Theorem}
\newtheorem {corfirstmain}[blaschke]{Corollary}
\newtheorem {secmain}{Theorem}[section]
\newtheorem {thred}[secmain]{Theorem}

\begin{document}
\title[Sub Hilbert Spaces]{Sub-Hardy Hilbert spaces on the circle and torus}

\author{Ajay Kumar}
\address{Department of Mathematics, University of Delhi, Delhi, 110007}
\email{nbkdev@gmail.com}

\author{Niteesh Sahni}
\address{Department of Mathematics, Shiv Nadar University, Chithera,\newline Tehsil Dadri 203207, Uttar Pradesh (India)}
\email{niteeshsahni@gmail.com}

\author{Dinesh Singh}
\address{Department of Mathematics, University of Delhi, Delhi, 110007}
\email{dineshsingh1@gmail.com}

\date{October 24, 2013}

\begin{abstract}
In \cite{sahsin}, Sahni and Singh settled a problem posed in \cite{yous} by generalizing the main result of \cite{yous} with a simple proof in the setting of sub-Hilbert spaces in $H^2(\mathbb{T})$. In this paper, we extend the main result of \cite{sahsin} to the setting of the Banach spaces $H^p$, $1\le p\le\infty$ on the circle and the torus. 
\end{abstract}

\maketitle

\section{Introduction}
Let $H$ be a Hilbert space contained in the Hardy space $H^{2}$. The inner
products on $H$ and $H^{2}$ are denoted by $\left\langle \cdot ,\cdot
\right\rangle _{H}$ and $\left\langle \cdot ,\cdot \right\rangle _{2}$
respectively. In \cite{yous} Yousefi and Hesameddini assume $H$ to satisfy
the following axioms:
\begin{itemize}
\item[\bf{A1.}] If  there are four functions $f_1$, $f_2$, $g_1$, $g_2\in H$ such that $\left\langle f_{1},g_{1}\right\rangle _{H}=\left\langle
f_{2},g_{2}\right\rangle _{H}$, then we have  $\left\langle
f_{1},g_{1}\right\rangle _{2}=\left\langle f_{2},g_{2}\right\rangle
_{2}$.
\item[\bf{A2.}] If $\varphi $ is any inner function, then\\ $\left\langle\varphi f,\varphi g\right\rangle _{H}=\left\langle f,g\right\rangle _{H}$
for all $f,g\in H$.
\end{itemize}
\noindent and pose the open question \lq\lq {\em Is every Hilbert space $H$ satisfying axioms $A1$ and $A2$ of the form $\varphi H^2$ for some $\varphi\in H^\infty$, and there exists a constant $k$ such that $\left\langle f,g\right\rangle _{H}=k\left\langle f,g\right\rangle _{2}$ for every $f,g\in H$}\rq\rq?

Recently, Sahni and Singh \cite{sahsin} have settled the above open problem in the affirmative. In fact, they prove a far more general result by assuming axiom $(A2)$ to hold for a fixed Blaschke factor of order $n$ instead of every inner function. The main result of \cite{sahsin} runs as follows:
\begin{Sahni}[Theorem $3.1$, \cite{sahsin}]\label{sahni}
Let $H$ be a Hilbert space contained in $H^2$ such that:
\begin{itemize}
\item[(i)]$T_B\left( H\right) \subset H$ and $\left\langle Bf,Bg\right\rangle _{H}=\left\langle f,g\right\rangle _{H}$ for all $f,g\in H$;
\item[(ii)] If  $f_1,f_2,g_1,g_2\in H$ satisfy $\\
\left\langle f_1,g_1\right\rangle _{H}=\left\langle
f_2,g_2\right\rangle _{H}$, then we have \newline
$\left\langle f_1,g_1\right\rangle _{2}=\left\langle f_2,g_2\right\rangle _{2}$.
\end{itemize}
Then there exist unique $B-$inner functions $b_1$, $b_2$, $\ldots$, $b_r$ $(r\le n)$ such that $$H=b_1H^2(B)\oplus\cdots\oplus b_rH^2(B)$$ and 
the $B-$ matrix of the $r-$ tuple $(b_1,\ldots,b_r)$ is $B-$ inner.
\end{Sahni}
The main result of \cite{yous} is the following invariant subspace characterization for the operator $S$ which stands for multiplication by the coordinate function $z$ on $H$ (that is $S:H\longrightarrow H$ such that $S(f(z))=zf(z)$): 
\begin{YouHesa}[Theorem 3, \cite{yous}]\label{orig1}
Let $H$ be a Hilbert space contained in $H^{2}$ satisfying axioms $A1$ and $%
A2$. Let $M$ be a closed subspace of $H$ that is invariant under the
operator $S$. Further if the set of multipliers of $H$ coincide with $%
H^{\infty }$, then there exists a unique inner function $\varphi $ such that $M=\varphi H$.
\end{YouHesa}
Theorem \ref{orig1} has also been improved in \cite{sahsin} by showing that the condition $M(H)=H^\infty$ is redundant. The proof relies on the solution of the above open problem which comes as a special case by taking $B(z)=z$ in Theorem \ref{sahni}. 

In the present paper we extend Theorem \ref{sahni}, and also the invariant subspace characterizations of \cite{sahsin} and \cite{yous} to the context of $H^p$ spaces for all $p\ge 1$. 

\section{Preliminary Results}
Let $\mathbb{D}$ denote the open unit disk, and its boundary, the unit circle by $\mathbb{T}$. The Lebesgue space $L^{p}$ on the unit circle is a
collection of complex valued functions f on the unit circle such that $\int {|f|^{p}}dm$ is finite, where $dm$ is the normalized Lebesgue measure on $\mathbb{T}$.

The Hardy space $H^p$ is the following closed subspace of $L^p$: 
\begin{equation*}
\left\{ f\in L^p : \int{f z^n} dm=0 ~\forall ~n\ge1 \right\} .
\end{equation*}

For $1\le p<\infty$, $H^p$ is a Banach space under the norm 
\begin{equation*}
\|f\|_p=\left(\int{|f|^p} dm\right)^\frac{1}{p} .
\end{equation*}

However $H^\infty$ is a Banach space under the norm 
\begin{equation*}
\|f\|_\infty= \inf\{K: m\{z\in T: |f(z)|>K\}=0\} .
\end{equation*}

The Hardy space $H^{2}$ turns out to be a Hilbert space under the inner
product 
\begin{eqnarray*}
<f,g>_2&=&\int {f\overline{g}}dm.
\end{eqnarray*}
An inner function is any function $\varphi$ in $H^2$ that satisfies $\left\langle \varphi f,\varphi g \right\rangle_{2}=\left\langle f, g \right\rangle_{2}$ for every $f,g\in H^2$. Equivalently, this means that $\|\varphi f\|_{2}=\|f\|_{2}$ for every $f\in H^2$.
By a finite Blaschke product $B(z)$ we mean
$$\prod_{i=1}^n\dfrac{z-\alpha_i}{1-\overline{\alpha_i} z}$$ where each $\alpha_i\in\mathbb{D}$. Each finite Blaschke factor is an inner function. The isometric operator of multiplication by $B$ on $H^2$ shall be denoted by $T_B$. When $n=1$ and $\alpha_1=0$ then $B(z)=z$. For the rest of this paper we shall assume that $\alpha_1=0$. This shall not cause any loss in generality due to the conformal invariance of $H^2$.
By $H^2(B)$ we denote the closed linear span of $\{B^n:n=0,1,\ldots\}$ in $H^2$. A detailed account of $H^{p}$ spaces can be found in \cite{Duren}, \cite{garnett}, and \cite{hoffman}. 

We now record results that will be used in our derivations. Note that throughout this paper $B(z)$ shall denote an arbitrarily chosen but then fixed finite Blaschke factor with $B(0)=0$.  

In \cite{sinth}, an orthonormal basis of $H^2$ is constructed in terms of the Blaschke factor $B(z)$. The description of the basis is $\left\{e_{j,m}:0\le j\le n-1, m=0,1,2,\ldots\right\}$, where $e_{j,m}=\dfrac{\sqrt{1-|\alpha_{j+1}|^2}}{1-\overline{\alpha_{j+1}}z}$. As a consequence, we can write 
\begin{equation}\label{h2decomp}
H^2=e_{0,0}H^2(B)\oplus e_{1,0}H^2(B)\oplus\cdots\oplus e_{n-1,0}H^2(B)
\end{equation}

Now, to each $r-$ tuple $(\varphi_1,\ldots,\varphi_r)$ of functions in $H^2$ we associate a $n\times r$ matrix of $H^2$ functions called the $B-$ matrix of $(\varphi_1,\ldots,\varphi_r)$ defined by $$A=(\varphi_{ij}),\text{ }0\le i\le n-1,\text{ and }1\le j\le r$$ where
$$\varphi_j=\sum_{i=0}^{n-1}e_{i,0}\varphi_{ij},\text{ }1\le j\le r$$
We say that $A$ is $B-$ inner if $$(\overline{\varphi_{ji}})(\varphi_{ij})=(\delta_{st})$$ where $1\le s,t\le r$, and $\delta_{st}$ is the Kronecker delta.

The following lemma provides a necessary and sufficient condition for a $B$- matrix to be $B$- inner.

\begin{singthuk}[Lemma $3.9$, \cite{sinth}]\label{lemm}
The $B-$ matrix of the $r-$ tuple $(\varphi_1,\ldots,\varphi_r)$ of $H^\infty$ functions is $B-$ inner if and only if $\{B^m\varphi_i:1\le i\le r, m=0,1,2,\ldots\}$ is an orthonormal set in $H^2$.
\end{singthuk}

A decomposition similar to equation (\ref{h2decomp}) is valid for $H^p$ spaces as well:
\begin{hpdecomp}[Lemma $2$, \cite{sahsin2}]\label{hpdecomp}
For $p\ge1$, we can write
\begin{equation*}
H^{p}=e_{0,0}H^{p}\left( B\right) \oplus e_{1,0}H^{p}\left( B\right)
\oplus \cdots \oplus e_{n-1,0}H^{p}\left( B\right),
\end{equation*}
where $H^p(B)$ stands for the closure (weak-star closure if $p=\infty$) of $span\{1,B,B^2,...\}$ in $H^p$, and $e_{j,m}$ are as defined above.
\end{hpdecomp}

We also borrow the following characterization of multipliers of $L^2$ into $L^p$ established in the proof of Corollary $5.1$ of \cite{sinagra}:

\begin{sinagra2}\label{sinagra2}
Let $1\le p\le 2$. Suppose $g\in L^q$, for some $q$ such that $g$ multiplies $L^2$ into $L^p$, then $g\in L^{\frac{2p}{2-p}}$. In other words, the set of multipliers of $L^2$ into $L^p$ is the space $L^{\frac{2p}{2-p}}$.
\end{sinagra2} 

The space $BMO$ is the space of $L^1$ functions $f$ such that 
$$\|f\|_{\ast}:=\sup_I\dfrac{1}{|I|}\int_{I}\left|f-\dfrac{1}{|I|}\int_{I}f \right|<\infty$$
where $I$ is any subarc of the unit circle $\mathbb{T}$, and $|I|$ denotes the normalized Lebesgue measure of $I$.

$BMO$ is a Banach space under the norm:
$$\|f\|:=\|f\|_{\ast}+|f(0)|.$$
Now $BMOA=BMO\cap H^1$. It is well known that $BMOA\subset H^p$ for all $p<\infty$. Some important invariant subspaces of $BMOA$ have been characterized in \cite{sahsin3}.

\section{One variable results}
\begin {blaschke}\label{lemm3}
Let $H$ be a Hilbert space contained in $H^p$. Suppose there are $H^\infty$ functions $\varphi_1,\ldots,\varphi_r$ such that whenever $i\neq j$, $\varphi_i H^2(B)\perp \varphi_j H^2(B)$ in the inner product of $H$. Then $r\le n$.
\end {blaschke}

\begin{proof}
The proof is on similar lines as that of Lemma $4.3$ in \cite{sinth}. We only outline the important steps here. For the sake of simplicity let $n=2$ (the general case follows on similar lines).

If possible let $\varphi_1,\varphi_2,\varphi_3$ be non zero $H^\infty$ functions such that $\varphi_i H^2(B)\perp \varphi_j H^2(B)$ whenever $i\neq j$. By Lemma \ref{hpdecomp}, for each $j=1,2,3$, there exist functions $\varphi_{0,j}$ and $\varphi_{1,j}$ in $H^\infty$ such that:
$$\varphi_j=e_{0,0}\varphi_{0,j}+e_{1,0}\varphi_{1,j}.$$ 
Now we form the matrix
$$
A=
\left(
\begin{array}{ccc}
\varphi_1&\varphi_2&\varphi_3\\
\varphi_{01}&\varphi_{02}&\varphi_{03}\\
\varphi_{11}&\varphi_{12}&\varphi_{13}
\end{array}
\right)
$$
A straight forward calculation shows that
\begin{equation}\label{det}
\det A=\varphi_1\lambda_1-\varphi_2\lambda_2+\varphi_3\lambda_3=0,
\end{equation}
where $\lambda_1=\varphi_{02}\varphi_{13}-\varphi_{03}\varphi_{12}$, $\lambda_2=\varphi_{01}\varphi_{13}-\varphi_{03}\varphi_{11}$, and $\lambda_3=\varphi_{01}\varphi_{12}-\varphi_{02}\varphi_{11}$.

Taking the inner product on both sides of equation (\ref{det}) with $\varphi_1\lambda_1$, we get $\varphi_1\lambda_1=0$. This implies that $\lambda_1=0$. In a similar fashion, $\lambda_2=\lambda_3=0$.

Using $\lambda_3=0$, it is easy to show that $\varphi_1\varphi_{02}-\varphi_2\varphi_{01}=0$. So $\varphi_1\varphi_{02}$ and $\varphi_2\varphi_{01}$ belong to $\varphi_1 H^2(B)\cap\varphi_2 H^2(B)$. This implies that $\varphi_{02}=\varphi_{01}=0$. Similarly, all other elements in second and third rows of $A$ are zero. This means that each $\varphi_j=0$, which is a contradiction. 
\end{proof}

\begin{firstmain}\label{main1}
Let $H$ be a Hilbert space contained in $H^p$ $(1\le p\le\infty)$ such that:
\begin{itemize}
\item[(i)]$T_B\left( H\right) \subset H$ and $\left\langle T_B f,T_B g\right\rangle _{H}=\left\langle f,g\right\rangle _{H}$ for all $f,g\in H$;
\item[(ii)] If  $f_1,f_2,g_1,g_2\in H$ satisfy $\\
\left\langle f_1,g_1\right\rangle _{H}=\left\langle
f_2,g_2\right\rangle _{H}$, then we have \newline
$\left\langle f_1,g_1\right\rangle _{2}=\left\langle f_2,g_2\right\rangle _{2}$.
\end{itemize}
Then 
\begin{itemize}
\item[a.] for $1\le p\le 2$, there exist unique $H^\infty$ functions $b_1$, $b_2$, $\ldots$, $b_r$ $(r\le n)$ such that $$H=b_1H^2(B)\oplus\cdots\oplus b_rH^2(B)$$ and further there exist scalars $k_1,\ldots,k_r$ such that $\|b_1 f_1+\cdots+b_r f_r\|_H^2=k_1\|f_1\|_2^2+\cdots+k_r\|f_r\|_2^2$ for all $f_1,\ldots,f_r\in H^2(B)$.\\
\item[b.] for $p>2$, $H=[0]$.
\end{itemize}
\end{firstmain}

\begin{proof}
Case $1$. $1\le p\le 2$. Since $T_B$ is an isometry on $H$, so by Wold decomposition \cite{wold}, we can write 
\begin{equation}\label{wold}
H=N\oplus T_B(N)\oplus T_B^2(N)\oplus\cdots,
\end{equation} 
where the wandering subspace $N$ is given by $H\ominus T_B(H)$. 
Choose $\varphi\in N$ with $\|\varphi\|_H=1$. We can see from equation (\ref{wold}) that the set $\{B^m\varphi:m=0,1,2,\ldots\}$ is orthonormal in $H$. Now in view of $(ii)$, we conclude that $\{B^m\varphi:m=0,1,2,\ldots\}$ is an orthogonal set in $H^2$. 

We claim that $\varphi H^2(B)\subset H$. Let $f$ be an arbitrary element of $H^2(B)$. Then $f=\sum_{k=0}^\infty \alpha_k B^k$. Let $f_m=\sum_{k=0}^m \alpha_k B^k$ so that $\|f_m-f\|_2\rightarrow 0$ as $m\rightarrow\infty$. We make the following computation:
\begin{eqnarray*}
\|\varphi f_m\|_H^2&=& \left\|\sum_{k=0}^m \varphi \alpha_k B^k\right\|_H^2\\
&=& \sum_{k=0}^m \left\|\varphi \alpha_k B^k\right\|_H^2\\
&=& \sum_{k=0}^m |\alpha_k|^2\left\|B^k\varphi \right\|_H^2\\
&=& \sum_{k=0}^m |\alpha_k|^2\\
&=& \|f_m\|_2^2.
\end{eqnarray*}
Since $\{f_m\}$ is a Cauchy sequence in $H^2$ so $\{\varphi f_m\}$ is a Cauchy sequence in $H$. But $H$ is complete, so there exists $g\in H$ such that $\varphi f_m\rightarrow g$ in the norm of $H$. Note that for any positive integer $k$, we can write:
$$\varphi f_m=(\alpha_0\varphi+\alpha_1 B\varphi+\cdots+\alpha_k B^k\varphi)+B^{k+1}\varphi h_m,$$ where $h_m=\alpha_{k+1}+\alpha_{k+2}B+\cdots+\alpha_m B^{m-k-1}$. Again $\{\varphi f_m\}$ is a Cauchy sequence in $H$, so $\{\varphi h_m\}$ is a Cauchy sequence in $H$. Thus  $\varphi h_m\rightarrow h$ in $H$. Hence
\begin{equation}\label{eq1}
g=(\alpha_0\varphi+\alpha_1 B\varphi+\cdots+\alpha_k B^k\varphi)+B^{k+1}h.
\end{equation}
From Lemma \ref{hpdecomp}, we can write 
\begin{equation}\label{eq2}
\varphi=\sum_{k=0}^{n-1}\sum_{m=0}^\infty \beta_{k,m}e_{k,m}.
\end{equation}
Substituting (\ref{eq2}) in equation (\ref{eq1}) we see that the $(j,m)$th Fourier coefficient of $g$ is equal to the $(j,m)$th Fourier coefficient of $\varphi f$. Hence $g=\varphi f$ and this establishes our claim. This also proves that $\varphi$ multiplies $H^2(B)$ into $H^p$. Hence, in view of equation (\ref{h2decomp}) we see that $\varphi$ multiplies $H^2$ into $H^p$. Because $L^2=H^2\oplus\overline{zH^2}$, we conclude that $\varphi$ multiplies $L^2$ into $L^p$. Now by Lemma \ref{sinagra2}, $\varphi\in L^{\frac{2p}{2-p}}$. This implies that $\varphi\in H^{\frac{2p}{2-p}}$.

Observe that $\frac{2p}{2-p}\ge 2$. So $\varphi\in H^2$. Taking $\psi:=\dfrac{\varphi}{\|\varphi\|_2}$, we find that $\psi\in H^2$ and the set $\{B^m\psi:m=0,1,2,\ldots\}$ is orthonormal in $H^2$.

Proceeding in a fashion similar to the proof of $\varphi H^2(B)\subset H$ it can be shown that $\psi H^2(B)\subset H^2$. Again in view of equation (\ref{h2decomp}), we find that $\psi$ multiplies $H^2$ back into $H^2$. This implies that $\psi\in H^\infty$ and hence $\psi$ is $B$-inner. This also implies that $\varphi\in H^\infty$ and so the wandering subspace $N$ is contained in $H^\infty$. 

Next we show that $\dim N\le n$. Let $\varphi_1,\ldots,\varphi_r$ be mutually orthogonal unit vectors in $H$. By continuity of inner product (on $H$), it is easily seen that $\varphi_i H^2(B) \perp \varphi_j H^2(B)$ whenever $i\neq j$. Therefore by Lemma \ref{lemm3} we have $r\le n$. 

Once again by the Wold type decomposition in equation (\ref{wold}) we see that 
\begin{eqnarray*}
H&=&\varphi_1 H^2(B)\oplus\cdots\oplus\varphi_r H^2(B)\\
&=&\|\varphi_1\|_2 b_1 H^2(B)\oplus\cdots\oplus \|\varphi_r\|_2 b_r H^2(B)\\
&=&b_1 H^2(B)\oplus\cdots\oplus b_r H^2(B)
\end{eqnarray*}
where $b_i=\frac{\varphi_i}{\|\varphi_i\|_2}$ are $B$-inner functions. This establishes the first part of the theorem.
Lastly, we derive an expression for the norm as follows:
\begin{eqnarray*}
\|b_1 f_1+\cdots+b_r f_r\|_H^2 &=& \|b_1 f_1\|_2^2+\cdots+\|b_r f_r\|_2^2\\
&=& \left\|\frac{\varphi_1}{\|\varphi_1\|_2} f_1\right\|_2^2+\cdots+\left\|\frac{\varphi_r}{\|\varphi_r\|_2}f_r\right\|_2^2\\
&=& k_1 \|\varphi_1 f_1\|_2^2+\cdots+k_r \|\varphi_r f_r\|_2^2\\
&=& k_1 \|f_1\|_2^2+\cdots+k_r \|f_r\|_2^2
\end{eqnarray*}
where $k_i=\frac{1}{\|\varphi_i\|_2^2}$, and $f_i\in H^2(B)$ for all $i=1,\ldots,r$.

Case $2$. $p>2$. We shall show that the wandering subspace in equation (\ref{wold}) is zero. If possible, let $0\neq b\in N$. Then by similar arguments as above it can be seen that $b$ multiplies $L^2$ into $L^p$. In this case $L^p\subset L^2$. So $b\in L^\infty$. But $b\in H^p$, therefore $b\in H^\infty$. Choose an $\epsilon>0$ such that the set $E=\{t:|b(t)|>\epsilon\}$ has positive measure (such an $\epsilon$ exists because $b$ is not zero). 

Since $L^p\subsetneq L^2$, so there exists $g\in L^2$ and $g\not\in L^p$. Define $h=\chi_E g$ so that $h\in L^2$ and $h\not\in L^p$ (for if $h\in L^p$, then $\int |g|^p = \int_{E} |\chi_E g|^p<\infty$, which is not possible).

Since $b$ multiplies $L^2$ into $L^p$, we have $bh\in L^p$. This together with the fact that $b$ is invertible on $E$ yields $$\int |h|^p=\int_{E} |h|^p=\int_{E} \dfrac{|h|^p |b|^p}{|b|^p}\le \dfrac{1}{\epsilon^p}\int_{E} |bh|^p<\infty.$$
This is a contradiction which is a consequence of the assumption that $b\neq0$. Hence $N=[0]$.
\end{proof}

\begin{corfirstmain}
Let $H$ be a Hilbert space contained algebraically in $BMOA$ and satisfying the conditions of Theorem \ref{main1}. Then $H=[0]$. 
\end{corfirstmain}

\begin{proof}
Since $BMOA\subset H^p$ for every finite $p$. So $H\subset H^p$ for all $p>2$. Hence, by Theorem \ref{main1}, we have $H=[0]$.
\end{proof}

\section{Two variable results}
First we give some preliminaries about two variable Hardy spaces. The torus is the cartesian product of two unit circles in the complex plane, and is denoted by $\mathbb{T}^2$. For $1\le p<\infty$, the Hardy space $H^p(\mathbb{T}^2)$ is the following closed subspace of $L^p(\mathbb{T}^2)$:
$$\left\{ f\in L^p(\mathbb{T}^2): \int_{\mathbb{T}^2} f(z,w) z^m w^n dm=0 \text{ whenever }m<0 \text{ or } n<0\right\}$$
where $dm$ is the normalized Lebesgue measure on the torus.\\

{\em Note: For the sake of convenience we shall write $H^p(\mathbb{T}^2)$ as $H^p$ hereafter in this section.}\\

The Hardy space $H^2$ is an Hilbert space with respect to the inner product:
$$<f,g>_2:=\int_{\mathbb{T}^2} f \bar{g} dm.$$

The collection $\{z^m w^n: m,n=0,1,\ldots \}$ is an orthonormal basis for $H^2$. A function $\varphi\in H^2$ is called an inner function if $\{z^m w^n \varphi: m,n=0,1,\ldots \}$ is an orthonormal set in $H^2$. Two operators $A$ and $B$ on a Hilbert space $H$ are called doubly commuting if $A^{\ast}$ commutes with $B$, or equivalently $A$ commutes with $B^{\ast}$.

The multiplication by the coordinate functions $z$ and $w$ will be denoted by $S_z$ and $S_w$ respectively. The operators $S_z$ and $S_w$ act as shifts on any Hilbert space $H$ contained in $H^p$ because $\cap_{n=0}^{\infty} S_z^n (H)=[0]$ and $\cap_{n=0}^{\infty} S_w^n (H)=[0]$.

We shall use the following result proved in \cite{red3} to prove the main result of this section:

\begin{thred}\label{orig}
Let $M$ be a Hilbert space contained in $H^p$, invariant under $S_z$ and $S_w$, and if $S_z$, $S_w$ are doubly commuting isometries on $M$, then  $$M=b H^2.$$
The function $b$ has the property that 
\begin{itemize}
\item[(i)] If $1\le p\le 2$ then $b\in H^{\frac{2p}{2-p}}$.
\item[(ii)] If $p>2$ then $b=0$.
\end{itemize}
Further, $\|b f\|_M=\|f\|_2$ for all $f\in H^2$.  
\end{thred}

The main result of this section is:

\begin{secmain}\label{main2}
Let $H$ be a Hilbert space contained in $H^p$, invariant under $S_z$ and $S_w$, and $S_z$, $S_w$ are doubly commuting isometries on $H$. Further, if the inner product on $H$ satisfies the condition:  
\begin{itemize}
\item[(P)] If  $f_1,f_2,g_1,g_2\in H$ such that $
\left\langle f_1,g_1\right\rangle _{H}=\left\langle
f_2,g_2\right\rangle _{H}$,\\ then
$\left\langle f_1,g_1\right\rangle _{2}=\left\langle f_2,g_2\right\rangle _{2}$.
\end{itemize}
Then 
\begin{itemize}
\item[(i)] for $1\le p\le 2$, there exists a unique inner function  $b\in H^\infty$ such that $$H=bH^2$$ and there exists a scalar $k$ such that $\|b f\|_H=k\|f\|_2$ for all $f\in H^2$.
\item[(ii)] for $p>2$, $H=[0]$.
\end{itemize}
\end{secmain}

\begin{proof}
{\bf The case $1\le p\le 2$}. Since $H$ satisfies the conditions of Theorem \ref{orig} we conclude that $H=\varphi H^2$ for some $\varphi\in H^{\frac{2p}{2-p}}$, and $\|\varphi f\|_M=\|f\|_2$ for all $f\in H^2$. Referring to the proof of Theorem \ref{orig} as given in \cite{red3}, we see that because $S_z$ and $S_w$ are doubly commuting shifts on $H$, the following Wold type decomposition is valid for $H$:
\begin{equation}\label{slocinski}
H=\sum\limits_{m=0}^\infty \sum\limits_{n=0}^\infty \oplus S_z^m S_w^n (N).
\end{equation}
Here $N$ is the wandering subspace $(H\ominus S_z(H))\cap(H\ominus S_w(H))$ and it turns out to be a one dimensional subspace generated by the unit vector $\varphi\in H$. Let us take $b:=\dfrac{\varphi}{\|\varphi\|_2}$.
From equation (\ref{slocinski}), it follows that the set
$$\left\{S_z^m S_w^n \varphi : m,n=0,1,2,\ldots \right\}$$
is orthonormal in $H$.

Now in view of property (P) satisfied by the inner product on $H$, we infer that the set 
$$\left\{S_z^m S_w^n b : m,n=0,1,2,\ldots \right\}$$
is orthonormal in $H^2$. This proves that $b$ is inner and $H=\varphi H^2=b H^2$. Also note that $\|b f\|_H=k \|f\|_2$ for all $f\in H^2$, where $k=\frac{1}{\|\varphi\|_2}$.

{\bf The case $p>2$}. Here $H=[0]$ is concluded directly from Theorem \ref{orig}.
\end{proof}

\end{document}